\documentclass[12pt]{amsart}
\usepackage{amsfonts,amssymb,amscd,amsmath,enumerate,verbatim,calc}
\usepackage{graphicx}

\newtheorem{thm}{Theorem}[section]
\newtheorem{lem}[thm]{Lemma}
\newtheorem{cor}[thm]{Corollary}
\newtheorem{prop}[thm]{Proposition}

\newtheorem{defn}[thm]{Definition}

\begin{document}
\title{ON FIXING SETS OF COMPOSITION AND CORONA PRODUCTS OF GRAPHS}
\author{I. Javaid*, M. S. Aasi, I. Irshad, M. Salman }
%\subjclass{Primary: , Secondary: }
\keywords{ Fixing set; Composition product of graphs; Corona product of graphs\\
\indent 2010 {\it Mathematics Subject Classification.} 05C25, 05C76\\
\indent $^*$ Corresponding author: imran.javaid@bzu.edu.pk}
%\indent This research of the authors was partially supported by the
%Higher Education Commission of\\
%\indent Pakistan}
\address{Centre for advanced studies in Pure and Applied Mathematics,
Bahauddin Zakariya University Multan, Pakistan\newline Email:
imran.javaid@bzu.edu.pk, mshahbazasi@gmail.com,
iqrairshad9344@gmail.com, solo33@gmail.com}
\date{}
\maketitle

\begin{abstract}
A fixing set $\mathcal{F}$ of a graph $G$ is a set of those vertices
of the graph $G$ which when assigned distinct labels removes all the
automorphisms from the graph except the trivial one. The fixing
number of a graph $G$, denoted by $fix(G)$, is the smallest
cardinality of a fixing set of $G$. In this paper, we study the
fixing number of composition product, $G_1[G_2]$ and corona product,
$G_1 \odot G_2$ of two graphs  $G_1$ and $G_2$ with orders $m$ and
$n$ respectively. We show that for a connected graph $G_1$ and an
arbitrary graph $G_2$ having $l\geq 1$ components $G_2^1$, $G_2^2$,
... $G_2^l,$ $mn-1\geq fix(G_1[G_2])\geq m\left(\sum
\limits_{i=1}^{l} fix(G_2^i )\right)$. For a connected graph $G_1$
and an arbitrary graph $G_2$, which are not asymmetric, we prove
that $fix(G_1\odot G_2)=m fix( G_2)$. Further, for an arbitrary
connected graph $G_{1}$ and an arbitrary graph $G_{2}$ we show that
$fix(G_1\odot G_2)= max\{fix(G_1), m fix(G_2)\}$.
\end{abstract}

\section{Introduction}
For readers convenience, we recall some basic definitions. A
\textit{graph} $G$ is an ordered pair $(V(G), E(G))$ consisting of a
set $V(G)$ of \textit{vertices} and a set $E(G)$ of \textit{edges}.
The number of vertices and edges of $G$ are called the
\textit{order} and the \textit{size} of $G$ respectively. If two
vertices $u$ and $v$ are joined by an edge then they are called
\emph{adjacent}, otherwise they are called \emph{non-adjacent}. The
\emph{open neighborhood} of a vertex $u$ is $N (u) = \{v \in V (G)
:$ $v$ is adjacent to $u$ in $G\}$ and the \emph{closed
neighborhood} of $u$ is $N [u] = N (u) \cup\{u\}$. For a subset $U$
of $V (G)$, the set $N_G (U ) = \{v \in V (G) :$ $v$ is adjacent to
some $u \in U \}$ is the open neighborhood of $U$ in $G$. The number
$|N(v)|$ is  called the degree of $v$ in $G$, denoted by $deg_G(v)$.
We simply write $deg(v)$ if the graph $G$ is clear from the context.
We denote the maximum degree of $G$ by $\Delta(G)$.
\par If $S$ is a nonempty set of vertices of a graph $G$ then the
\emph{subgraph of $G$ induced by $S$} is the induced subgraph with
vertex set $S$, where the induced subgraph $H$ of a graph $G$ is the
subgraph of $G$ such that whenever $u$ and $v$ are vertices of $H$
and $uv$ is an edge of $G$ then $uv$ is an edge of $H$ as well.

A \emph{graph homomorphism} is a mapping between vertex sets of two
graphs that respects their structure. More concretely, it maps
adjacent vertices to adjacent vertices. A bijective graph
homomorphism is called \emph{isomorphism} and an isomorphism from a
graph $G$ onto itself is called an \emph{automorphism} of $G$. The
collection of all automorphisms of a graph $G$ forms a group under
the operation of composition of functions, denoted by $\Gamma(G)$.

A set $\mathcal{D}\subseteq V(G)$ is called a \emph{determining set}
if whenever $g,h\in \Gamma(G)$ such that $g(v)=h(v)$ for all $v\in
\mathcal{D}$, then $g(u)=h(u)$ for every $u\in V(G)$, i.e., every
automorphism is uniquely determined by its action on the vertices of
$\mathcal{D}$. The determining number of a graph is the size of a
smallest determining set, denoted by $Det(G)$. Determining sets of
graphs were introduced by Boutin in \cite{DLB}. She gave several
ways of finding and verifying determining sets. The natural lower
bounds on the determining number of some graphs were also given.
Determining sets are frequently used to identify the automorphism
group of a graph. For further work on determining sets and its
relation with other parameters see \cite{DLB,JC}.

\par The orbit of a vertex $w$ is the set $\mathcal{O}(w)\subseteq
V(G)$ such that for every $v\in \mathcal{O}(w)$ there exists an
$\alpha \in \Gamma (G)$ such that $\alpha(w)=v$ and we say that $w$
is $similar$ to $v$, denoted by $w \sim v$. An automorphism $\alpha
\in \Gamma(G)$ is said to \emph{fix} a vertex $w \in V (G)$ if
$\alpha(w) = w$. The \textit{stabilizer} of a vertex $w$ in a graph
$G$ is the set of all automorphisms of $G$ that fixes $w$ and is
denoted by $\Gamma_w(G)$. For $\mathcal{F}\subset V(G)$, an
automorphism $\alpha$ is said to fix the set $\mathcal{F}$ if for
every $w\in \mathcal{F}$, we have $\alpha(w)=w$. The set of
automorphisms that fix $\mathcal{F}$ is a subgroup
$\Gamma_{\mathcal{F}}(G)$ of $\Gamma(G)$ and
$\Gamma_{\mathcal{F}}(G)$ = $\bigcap\limits_{w\in
\mathcal{F}}\Gamma_w (G)$. A set $\mathcal{F}$ is called a fixing
set of $G$ if $\Gamma_{\mathcal{F}}(G)=\{e\}$, where $\{e\}$ is the
identity automorphism. In \cite{FH}, Erwin and Harary introduced an
equivalent concept of the determining number called the \emph{fixing
number} of a graph $G$, $fix(G)$, which is defined as the minimum
cardinality of a fixing set of $G$. A fixing set containing $fix(G)$
number of vertices is called a \emph{minimum fixing set} of $G$. A
vertex $x\in V (G)$ is called a \emph{fixed vertex} if $g(x) = x$
for all $g \in \Gamma(G)$, i.e., $\Gamma_x (G) = \Gamma(G)$. A
vertex $x \in V (G)$ is said to fix a pair $(u, v)$ of similar
vertices, if $h(u)\neq v$ or $h(v) \neq u$ whenever $h\in \Gamma_x
(G)$. For a pair $(u,v)$ of similar vertices $fix(u, v)(= fix(v, u))
= \{x \in V (G) : g(u) \neq v$ and $g(v) \neq u$ for all $g \in
\Gamma_x (G)\}$ is called the fixing set relative to the pair $(u,
v)$ \cite{B}. A graph $G$ is said to be \emph{asymmetric} if
automorphism group $\Gamma(G)$ consists of just the identity
element. All graphs considered in this paper are simple, non-trivial
and have non-trivial automorphism group unless otherwise stated.

Fixing sets of graphs were further studied by Harary and Erwin in
\cite{FD} and it was noted that for any positive integer $n\geq 1$,
$fix(K_n)=n-1$, $fix(P_n)=1$ for $n\geq 2$ and $fix(C_n)=2$ for
$n\geq 3$. The problem of distinguishing vertices of a graph has
been studied using two approaches as mentioned in \cite{FD}. The
first approach involves the concept of metric dimension/location
number, introduced separately by Harary and Melter \cite{t28} and by
Slater \cite{t50}.
%is defined in the following way: For $u, v\in
%V(G)$,
% the number $d(u,v)$ denotes the length of a shortest path between $u$ and $v$,
%  and is called the {\it distance} between $u$ and $v$.
 %  Let $R = \{w_1, w_ 2,\ldots, w_k \}$
 %  be an ordered subset of $V(G)$. The \emph{representation} of a vertex $v$ of $G$
%with respect to $R$ is defined as the $k$-tuple $r (v|R ) = (d (v,w_
%1 ) , d(v, w_ 2 ), . . . , d (v, w_ k ))$. The set $R$ is called a
%resolving/locating set of $G$ if every two distinct vertices $x, y
%\in V (G)$ satisfy $r (x|R ) \neq r (y|R )$. The minimum cardinality
%of a resolving/locating set for $G$, denoted by $\beta(G)$/$loc(G)$,
%is called the {\it metric dimension}/{\it location number} of $G$
%\cite{t28,t50}. This notion has many application in several areas
%including coin weighing problems, network discovery and
%verification, robot navigation, connected joins in graphs, and
%strategies for Mastermind games.
The second approach uses the notion of symmetry breaking that was
formalized by Albertson and Collins \cite{MA} and independently by
Harary \cite{FMS,FH}. In this approach, a subset of the vertex set
is colored in such a way that all the automorphisms of the graph
result in identity automorphism. This approach leads to the idea of
fixing sets of graphs. In \cite{FD}, Harary and Erwin gave upper
bounds for the fixing number of a graph in terms of the number of
orbits under the action of $\Gamma(G)$ and in terms of the order of
$\Gamma(G)$. The notion of fixing set has its application to the
problem of programming a robot to manipulate objects \cite{KL}.

%In \cite{CRG}, Gibbons and Laison proved that a set of vertices is a
%fixing set if and only if it is a determining set. Considering the
%fact that the metric dimension is greater than or equal to the
%fixing number and automorphisms preserve distances, the metric
%dimension and the fixing number are closely related notions
%\cite{JC, FD} . C\'{a}ceres et al. studied this relation and
%answered the following question which appeared first in \cite{JC}:
%Can the difference between both parameters of a graph of order $n$
%be arbitrarily large?

Metric dimension of composition product (also called lexicographic
product) and corona product of graphs was studied in \cite{BM,SB}
and in \cite{JA}, respectively. Motivated by the close relationship
between the metric dimension and the fixing number of graphs(given
in \cite{FD}), in this paper, we consider the composition product
and the corona product of graphs in the context of fixing number.

 This paper consists of three sections including the introduction.
 In section 2 and 3, we give several results related to fixing sets
 and the fixing number of composition and corona product of graphs.
\section{Composition Product}
The \textit{composition product} of two graphs $G_1$ and $G_2$,
denoted by $G_1[G_2]$, is the graph with vertex set $V(G_1)\times
V(G_2)$ = $\{(a,v)\ |\ a\in V(G_1)$ and $v\in V(G_2)\}$,  where
$(a,v)$ is adjacent to $(b,w)$ whenever $ab\in E(G_1)$, or $a=b$ and
$vw\in E(G_2)$ \cite{C}. For any vertex $a\in V(G_1)$ and $b\in
V(G_2)$, we define the vertex set $G_2(a)$ = $\{(a,v)\in
V(G_1[G_2])\ |\ v\in V(G_2)\}$ and $G_1(b)$ = $\{(v,b)\in
V(G_1[G_2])\ |\ v\in V(G_1)\}$. Let $G_1$ and $G_2$ be two
non-trivial graphs containing $k\geq 1$ components $G_1^1$,
$G_1^2$,\ldots,$G_1^k$ and $l\geq 1$ components $G_2^1$,
$G_2^2$,\ldots,$G_2^l$ respectively with $|V(G_2^j)|\geq 2$ for each
$j=1,2,\ldots,l$. For $a^i\in V(G_1^i)$ and $1\leq i \leq k$, we
define the vertex set $G_2^j(a^i)$ = $\{(a^i,v)|v\in V(G_2^j)\}$.
From the definition of $G_1[G_2]$, it is clear that for every $(a,
v)\in V(G_1[G_2])$, $deg(a, v)= deg_{G_1}(a)\cdot |V(G_2)|+
deg_{G_2}(v)$ in $G_1[G_2]$. If $G_1$ is a disconnected graph having
$k\geq 2$ components $G_1^1$, $G_1^2$, \ldots, $G_1^k$, then
$G_1[G_2]$ is also a disconnected graph having $k$ components such
that $G_1[G_2]=G_1^1[G_2]\cup G_1^2[G_2]\cup \ldots \cup G_1^k[G_2]$
and each component $G_1^i[G_2]$ is the composition product of
connected component $G_1^i$ of $G_1$ with $G_2$, therefore
throughout the paper, we will assume $G_1$ to be connected. Using a
result on the fixing number of a disconnected graph given in
\cite{az} and the fixing number of components of $G_1[G_2]$, we give
a general formula for fixing number of $G_1[G_2]$. Some useful
results related to the structure and distance properties of
composition product of two graphs are stated here:

\begin{thm}\cite{az}
Let $G$ be a graph having $k\geq 2$ components $G_1$, $G_2$, ...,
$G_k$ with $|G_j|\geq 2$ for all $1\leq j\leq k$. Let $H$ be the
subgraph of $G$ having those components of $G$ which are not
asymmetric, and $H_1,H_2,\ldots,H_l$ are the subgraphs of $G$ having
$m_1,m_2,\ldots,m_l$ asymmetric components of $G$ such that for
$G_i,G_j\in H_r$, $r=1,2,\ldots,l$, $G_i\cong G_j$, then
\[fix(G)=\sum \limits_{G_i\in H} fix(G_i)+\sum\limits_{i=1}^{l}m_i
-l.\]
\end{thm}
Let G be a connected graph, then the distance between two vertices
$u$ and $v$ in $G$, $d_G(u,v)$, is the length of a shortest $u-v$
path in G.
\begin{prop}\cite{WI}
Suppose $(a, b)$ and $(a' , b')$ are two vertices of $G_1[G_2]$.
Then
$$d_{G_1[G_2]} ((a, b), (a' , b'))=\left\{\begin{array}{ll}
       d_{G_1}(a,a') & \,\,\,\,\,\,\, \mbox{if}\,\ a\neq a',\\
       d_{G_2}(b,b') & \,\,\,\,\,\,\, \mbox{if}\,\ a=a' $and$\,\ deg_{G_1}(a)=0,\\
       min\{d_{G_2}(b,b'),2\} & \,\,\,\,\,\,\, \mbox{if}\,\ a=a' $and$\,\ deg_{G_1}(a)\neq 0.
            \end{array}
             \right.
 $$
\end{prop}
\begin{thm}\label{thm01}\cite{WI}
Let $G_1[G_2]$ be the composition product of two graphs $G_1$ and
$G_2$. For every $\alpha \in \Gamma(G_1)$ and $\beta \in
\Gamma(G_2)$, there exist automorphisms $f_\alpha$ and $f_\beta$ on
$G_1[G_2]$ given by
 \begin{center}
 $f_{\alpha}(x,i)=(\alpha(x),i)$, $\forall$ $\alpha \in \Gamma
 (G_1)$,
 \\
  $f_{\beta}(x,i)=(x,\beta(i))$, $\forall$ $\beta \in \Gamma (G_2)$.
\end{center}
\end{thm}
\begin{prop}\label{pro1.4}\cite{SB}
Let $G_1[G_2]$ be the composition product of two graphs $G_1$ and
$G_2$. For two different vertices $a$ and $b$ of $G_1$, every two
different vertices $u,v\in G_2(a)$ satisfy
$d_{G_{1}[G_{2}]}(u,z)=d_{G_{1}[G_{2}]}(v,z)$ whenever $z\in
G_2(b)$.
\end{prop}

\begin{prop}\label{pro1.5}\cite{SB}
Let $G_1[G_2]$ be the composition product of two graphs $G_1$ and
$G_2$ with $G_2$ having $l\geq 1$ components $G_2^1, G_2^2,\ldots,
G_2^l$.
 For $a\in V(G_1)$ and $i, j\in \{1,2,...,l\}$ with $i\neq
j$, every two different vertices $x, y\in G_2^i(a)$ satisfy
$d_{G_{1}[G_{2}]}(x, z) = d_{G_{1}[G_{2}]}(y, z)$ whenever $z\in
G_2^j(a)$.
\end{prop}

Let $G$ be a graph. The number $e(v) = \max \limits_{u\in V(G)}
d_G(u,v)$ denotes the {\it eccentricity} of $v$. For $1\leq i\leq
e(v)$, the $i$th distance neighborhood of $v$ is $N_i(v) = \{u\in
V(G) | d_G(u,v) = i\}$.

\par From Theorem 2.3, observe that to study the automorphisms of $G_1[G_2]$, we
must keep in mind the automorphisms of both $G_1$ and $G_2$. We note
that for any two graphs $G_1$ and $G_2$ if $a\in \mathcal{O}(b)$ for
two distinct vertices $a, b\in V(G_1)$ then $(a, i)\in
\mathcal{O}(b, i)$ for all $i\in V(G_2)$, where $(a, i),(b, i)\in
V(G_1[G_2])$. As we know that for $a\in \mathcal{O}(b)$,
$deg_{G_1}(a)=deg_{G_1}(b)$ and for any $(a, i)\in V(G_1[G_2])$,
$deg(a, i) = deg_{G_1}(a)\cdot |V(G_2)|+ deg_{G_2}(i)$. So, $deg(a,
i) = deg(b, i)$ in $G_1[G_2]$ for all $i\in V(G_2)$. Suppose $(a,
i)\not\in \mathcal{O}(b, i)$, then there exists a vertex $(c, i)$ in
$N_{s}(a,i)$, where $1\leq s \leq e(a)$, and a vertex $(d, i)\in
N_{s}(b,i)$, where $1\leq s \leq e(b)$, such that $deg(c, i)\neq
deg(d, i)$ in $G_1[G_2]$. But then $deg(c)\neq deg(d)$, where $c\in
N_s(a)$ and $d\in N_s(b)$, which contradicts the fact that $a\in
\mathcal{O}(b)$. Hence $(a, i)\in \mathcal{O}(b, i)$ for all $i\in
V(G_2)$. Similarly, for two distinct vertices $i,j\in V(G_2)$ if
$j\in \mathcal{O}(i)$ so, $deg(j) = deg(i)$ in $G_2$. Then for any
$a\in V(G_1)$, $deg(a,i) = deg(a, j)$. Also, $(a, k)$ is adjacent to
$(a, i)$ in $G_1[G_2]$ if and only if $k$ is adjacent to $i$ in
$G_2$. So, the degree sequence of neighborhood of $(a, i)$ must be
same as of $(a, j)$. For otherwise, $j\notin \mathcal{O}(i)$ and
hence we conclude that $(a, j)\in \mathcal{O}(a, i)$. Hence, we have
the following result:

\begin{lem}\label{lem2}
The following assertions hold in the composition product,
$G_1[G_2]$, of two graphs $G_1$ and $G_2$.
\\
(i) If $a\in \mathcal{O}(b)$ in $G_1$, then $(a, i)\in
\mathcal{O}(b, i)$ in $G_1[G_2]$ for all $i\in V(G_2)$.
\\
(ii) If $j\in \mathcal{O}(i)$ in $G_2$, then $(a, j)\in
\mathcal{O}(a, i)$ in $G_1[G_2]$ for all $a\in V(G_1)$.
\end{lem}

\begin{lem}\label{lem1}\cite{B}
Let $G$ be a graph. If there exists an automorphism $\alpha$ of $G$
such that $\alpha(u)=v$, where $u, v \in V(G)$ with $ u\neq v$, and
$d_G(u,x)=d_G(v,x)$ for some $x\in V(G)$, then $x\notin fix(u,v)$.
\end{lem}
From Propositions \ref{pro1.4} and \ref{pro1.5} and Lemma
\ref{lem1}, we have the following two lemmas:
\begin{lem}\label{lma s1}
Let $G_1[G_2]$ be the composition product of two graphs $G_1$ and
$G_2$. For a pair of distinct vertices $a,b \in V(G_1)$, if $z\in
G_2(b)$, then $z\not\in fix(x,y)$ for all  $x, y \in G_2(a)$.
\end{lem}
\begin{lem}\label{lma s2}
Let $G_1[G_2]$ be the composition product of two graphs $G_1$ and
$G_2$ with $G_2$ having $l\geq 1$ components $G_2^1, G_2^2,\ldots,
G_2^l$. Then for $a\in V(G_1)$ and $x\in G_2^j(a)$, $x\not\in
\mathcal{F}(G_2^i(a))$, $1\leq i, j\leq l$ and $i\neq j$.
\end{lem}

%\begin{proof}
%Consider $y,z\in G_2^i(a)$ then, by Proposition \ref{pro1.5}, for
%all $x\in G_2^j(a)$ with $i\neq j$ we have $d(x,y)=d(x,z)$, so by
%Lemma \ref{lem1}, we have $x\not\in \mathcal{F}(G_2^i(a))$.
%we have v\in V(G_2^i)$ and
%$\mathcal{O}(v)=\{v, v_1, ..., v_l\}$. Then $deg_{G_2}(v) =
%deg_{G_2}(v_p)$ for all $p\in\{1,2,...l\}$. Also for any $(a, v)\in
%V(G_1[G_2])$, $deg(a, v) = deg_{G_1}(a)|V(G_2)|+deg_{G_2}(v)$
%implies that $deg(a,v) = deg(a,v_p)$ in $G_1[G_2]$. Since, for any
%$z\in G_2^j(a)$, $d(x, z) = d(y,z)$ for every $x,y\in G_2^i(a)$.
%Therefore, $x\not\in \mathcal{F} (G_2^i(a))$.
%\end{proof}

\begin{thm}\label{thm 01}
Let $G_1[G_2]$ be the composition product of two graphs $G_1$ and
$G_2$. If $\mathcal{F}$ is a minimum fixing set for $G_1[G_2]$. Then
for $a\in V(G_1)$, $\mathcal{F}(a)=\mathcal{F}\cap G_2(a)\neq
\emptyset$ and $\mathcal{F}(a)$ is a fixing set for $G_2(a)$.
\end{thm}

\begin{proof}
Suppose $\mathcal{F}(a)=\emptyset$. As $G_1[G_2]$ is not asymmetric,
so $\mathcal{F}\neq \emptyset$. By Lemma \ref{lma s1}, any $z\in
V(G_1[G_2])\setminus G_2(a)$ will not fix the vertices of $G_2(a)$.
Since $G_2$ is not asymmetric, there exist at least two vertices
$x,y\in G_2(a)$ such that $x \sim y$, a contradiction. Hence,
$\mathcal{F}(a)$ is non-empty. Moreover, $\mathcal{F}(a)$ is a
fixing set for $G_2(a)$, otherwise $\mathcal{F}$ is not a fixing set
for $G_1[G_2]$.
\end{proof}

\begin{thm}\label{thm 02}
Let $G_1[G_2]$ be the composition product of two graphs $G_1$ and
$G_2$ with $G_2$ having $l\geq 1$ components $G_2^1, G_2^2,\ldots,
G_2^l$ with each $G_2^i$ is not asymmetric. If $\mathcal{F}$ is a
minimum fixing set for $G_1[G_2]$. Then for $a\in V(G_1)$,
$\mathcal{F}_{i}(a) = \mathcal{F}\cap G_2^i(a)\neq \emptyset$ and
$\mathcal{F}_i(a)$ is a fixing set for $G_2^i(a)$. Moreover, if
$\mathcal{F}_i$ is a minimum fixing set for $G_2^i$, then
$|\mathcal{F}_i(a)|\geq |\mathcal{F}_i|$.
\end{thm}

\begin{proof}
\par Suppose $\mathcal{F}_i(a)=\emptyset$. By Lemma \ref{lma s2},
 any $z\in V(G_1[G_2])\setminus G_2^i(a)$ will not fix
the vertices of $G_2^i(a)$. Since $G_2^i$ is not asymmetric for each
$1\leq i \leq l$, so there exist vertices $x,y\in G_2^i(a)$ such
that $x \sim y$, a contradiction. Hence, $\mathcal{F}_i(a)$ is
non-empty. Moreover, $\mathcal{F}_i(a)$ is a fixing set for
$G_2^i(a)$, otherwise $\mathcal{F}$ is not a fixing set for
$G_1[G_2]$.

\par Now, we have to show that $|\mathcal{F}_i(a)|\geq |\mathcal{F}_i|$. Let
$\mathcal{F}_i(a)=\{(a, u_1), (a, u_2),...,$ $(a, u_t)\}$, where $t<
|\mathcal{F}_i|$, and let $S=\{u_1, u_2,..., u_t\}$ be a subset of
$V(G_2^i)$. Since $t<|\mathcal{F}_i|$, there exist at least two
distinct vertices $x,y\in V(G_2^i)\setminus S$ such that $x\sim y$.
Now, $(a, x), (a, y)\in G_2^i(a)$, $deg(a, x) = deg(a, y)$ and also
$(a, x)\sim (a, y)$ in $G_1[G_2]$, which contradicts the fact that
$\mathcal{F}$ is a fixing set for $G_1[G_2]$. Hence, $t\geq
|\mathcal{F}_i|$.
\end{proof}

\begin{lem}\label{lma s3}
Let $G_1$ and $G_2$ be two graphs with $G_{2}$ having $l\geq 1$
components $G_2^1, G_2^2, \ldots, G_2^l$. Let $\mathcal{F}$ be a
fixing set for $G_1[G_2]$ and  $a\in V(G_1)$. If $\mathcal{F}(a) =
\mathcal{F}\cap G_2(a)$, then $|\mathcal{F}(a)| = \sum
\limits_{i=1}^{l} fix(G_2^i )$.
\end{lem}

\begin{proof}
Let $\mathcal{F}_i$ be a fixing set of $G_2^i$ for $i\in \{1, 2,
..., l \}$. Now, consider a vertex set $\mathcal{F'}=
\bigcup\limits_{1\leq i \leq l} \mathcal{F}_i (a)$, where
$\mathcal{F}_i (a) = \{(a, x)\ |\ x \in \mathcal{F}_i\}$. Then
choose $\mathcal{F}(a) = \mathcal{F'}$. Since $\mathcal{F}_i$ is a
fixing set of $G_2^i$ for $1 \leq i \leq l$, so $\mathcal{F}_i(a)$
fixes vertices in $G_2^i (a)$, which implies that $\mathcal{F}(a)$
fixes the vertices in $G_2(a)$. Thus $|\mathcal{F}(a)| = \sum
\limits_{i=1}^{l} fix(G_2^i)$.
\end{proof}

\par To prove a result for the bounds of fixing number of composition product of two
graphs, we first find fixing number of some simple families of
graphs. For instance, consider $G_1 = P_m$ be a path graph of order
$m\geq 2$ and $G_2$ be a graph having $l \geq 2$ components $G_2^i\
(1\leq i\leq l)$ each of which is a path graph $P_n$ of order $n\geq
2$. Since each component of $G_2^i$ is a path, so $fix(G_2^i) = 1$
implies that $\sum\limits_{i=1}^{l} fix(G_2^i) = l$. Also, to fix
each $G_2^i(a)$ in $G_1[G_2]$ for every $a\in V (G_1)$, we need to
fix one vertex of $G_2^i(a)$. In this way, we need to fix $l$
vertices to fix $G_2(a)$ for $a\in V (G_1)$. Thus, to fix
$G_1[G_2]$, we have to fix $ml$ vertices of $G_1[G_2]$ so the
$fix(G_1[G_2])=ml$. Now consider $G_1 = K_2$ be a complete graph and
$G_2 = K_n\ (n\geq 2)$ be a complete graph. Then in this case,
$G_1[G_2]$ is a complete graph with $2n$ vertices, so we have
$fix(G_1[G_2]) = 2n - 1 = 2(n - 1) + 1 = 2fix(G_2) + 1$.
\begin{thm} For any connected graph $G_1$ and for an arbitrary graph
$G_2$ having $l\geq 1$ components $G_2^1$, $G_2^2$, ... $G_2^l,$
$$mn-1\geq fix(G_1[G_2])\geq m\left(\sum \limits_{i=1}^{l} fix(G_2^i
)\right)$$ where $m$, $n$ are orders of $G_1$ and $G_2$,
respectively.
\end{thm}
\begin{proof}For upper bound, consider $G_1$ and $G_2$ both complete
then $G_1[G_2]$ is also complete so $fix(G_1[G_2])=mn-1$. Now, to
find the lower bound, we consider the following cases:\\ \noindent
Case $1$: For $l=1$, $G_2$ is connected and using the Lemma \ref{lma
s1} and Theorem \ref{thm 01}, we have $fix(G_1[G_2])= mfix(G_2)$.\\
Case $2$: For $l\geq 2$ suppose that no two asymmetric components of
$G_2$ are isomorphic then clearly $fix(G_1[G_2])= m\left(\sum
\limits_{i=1}^{l} fix(G_2^i )\right)$ and if $G_2$ has asymmetric
components $G^i_2$ and $G_2^j$ such that $G^i_2\cong G_2^j$ then
$fix(G_1[G_2])\geq m\left(\sum \limits_{i=1}^{l} fix(G_2^i )\right)$
because there exists $\alpha\in \Gamma(G_2)$ with $\alpha
(V(G_2^i))=V(G_2^j)$ then the map $\gamma= (i,\alpha)$ is an
automorphism of $G_1[G_2]$.
\end{proof}
\begin{defn}
 Let $G_1[G_2]$ be a composition product of two graphs $G_1$ and
$G_2$. We refer the projections $p_{G_1}$ and $p_{G_2}$ as the
corresponding $G_1$-coordinate or $G_2$- coordinate. That is
$p_{G_1}(u,v)=u$ and  $p_{G_2}(u,v)=v$ for $(u,v)\in V(G_1[G_2])$.
\end{defn}
%\begin{thm}\label{thm 04}
% If $G_1$ and $G_2$ are asymmetric graphs, then $G_1[G_2]$ is asymmetric.
%\end{thm}
%
%\begin{proof}
%\par Let $u_a$ be the projection of all vertices of $G_2(a)$. Let $Y$ be
%a graph where $V(Y)=\{u_a| a\in V(G_1)\}$ and $u_au_b\in E(Y)$
%whenever $ab\in E(G_1)$. So $Y$ is isomorphic to $G_1$, since $G_1$
%is asymmetric so $Y$ is also asymmetric hence no vertex of $G_2(a)$
%is in orbit of any vertex of $G_2(b)$ for $b\in V(G_1)$ and $b\neq
%a$. Since $G_2$ is asymmetric, so by Lemma \ref{lma s1} and by Lemma
%\ref{lem2} (ii) any $i, j\in V(G_2)$ $i\notin \mathcal{O}(j)$
%implies that $\mathcal{O}(a, i)\notin (a, j)$ i.e. $G_2(a)$ is also
%asymmetric for any $a \in V(G_1)$. Hence $G_1[G_2]$ is asymmetric.
%\end{proof}
%
%\begin{thm}\label{thm 05}
% If $G_2$ is asymmetric graph then $fix(G_1[G_2])=fix(G_1)$.
%\end{thm}
%
%\begin{proof}
%\par Since $G_2$ is asymmetric graph so there will be no automorphism of
%$G_1[G_2]$ which maps any vertex of $G_2(a)$ to its own vertex. Let
%$u_a$ be the projection of all vertices of $G_2(a)$ and $Y$ be a
%graph where $V(Y)=\{u_a| a\in V(G_1)\}$ and $u_au_b\in E(Y)$
%whenever $ab\in E(G_1)$. So $Y$ is isomorphic to $G_1$ hence $u_a\in
%\mathcal{O}(u_b)$ in $Y$ if and only if $a\in \mathcal{O}(b)$ in
%$G_1$ so $fix(Y)=fix(G_1)$. Hence $fix(G_1[G_2])= fix(G_1)$.
%\end{proof}
\begin{lem}\label{thm 06}
If $G_{1}$ and $G_{2}$ are asymmetric, then $G_1[G_2]$ is also
asymmetric.
\end{lem}
\begin{proof} Let $u_a$ be the projection of all the vertices of $G_2(a)$. Let $Y$
be a graph with $V(Y)=\{u_a\ |\ a\in V(G_1)\}$ and $u_au_b\in E(Y)$
whenever $ab\in E(G_1)$. Thus $Y$ is isomorphic to $G_1$. Since
$G_1$ is asymmetric, so $Y$ is also asymmetric and hence no vertex
of $G_2(a)$ is in the orbit of any vertex of $G_2(b)$ for $b\in
V(G_1)$, where $b\neq a$. Since $G_2$ is asymmetric, so by Lemma
\ref{lem2}(ii), for distinct $i, j\in V(G_2)$, $i\not\in
\mathcal{O}(j)$ implies that $\mathcal{O}(a, i)\not\in
\mathcal{O}(a, j)$. That is, $G_2(a)$ is also asymmetric for any $a
\in V(G_1)$. Hence, $G_1[G_2]$ is asymmetric.
\end{proof}
\begin{lem}\label{thm 061} Let $G_{1}$ and $G_{2}$ be two graphs
with $G_{2}$ is asymmetric. Then $fix(G_1[G_2])= fix(G_1)$.
\end{lem}
\begin{proof}Since $G_2$
is asymmetric graph, so no automorphism of $G_1[G_2]$ maps vertices
of $G_2(a)$ to vertices of $G_2(b)$ for distinct $a, b \in V(G)$.
Let $u_a$ be the projection of all the vertices of $G_2(a)$ and $Y$
be a graph with $V(Y)=\{u_a\ |\ a\in V(G_1)\}$ and $u_au_b\in E(Y)$
whenever $ab\in E(G_1)$. Thus $Y$ is isomorphic to $G_1$, and hence
$u_a\in \mathcal{O}(u_b)$ in $Y$ if and only if $a\in
\mathcal{O}(b)$ in $G_1$, which implies that $fix(Y)=fix(G_1)$.
Hence, $fix(G_1[G_2])= fix(G_1)$.
\end{proof}
%\begin{thm}\label{thm 06}
%Let $G_1$ be a graph containing $k\geq 1$ components $G_1^1$,
%$G_1^2$, ... $G_1^k$ with $|V(G_1^i)|\geq 2$ for each
%$i=1,2,3,...,k$ and $G_2$ be a graph containing $l\geq 1$ components
%$G_2^1$, $G_2^2$, ... $G_2^l$ with $|V(G_2^j)|\geq 2$. Then we have
%$fix(G_1[G_2])= \sum \limits_{i=1}^{k} |V(G_1^i)| fix(G_2)$
%\end{thm}

%\goodbreak

%%% ----------------------------------------------------------------------

\section{Corona Product of Graphs}

Let $G_1$ and $G_2$ be two graphs with $|V(G_1)|= m$ and
$|V(G_2)|=n$. Their \emph{corona product} is denoted by $G_1\odot
G_2$, is the graph obtained from $G_1$ and $G_2$ by taking one copy
of $G_1$ and $m$ copies of $G_2$ and joining each vertex from the
$i^{th}$-copy of $G_2$ by an edge with the $i^{th}$-vertex of $G_1$.
Let $u_{i}\in V(G_1)$ then $C_{i}= (V_{i}, E_{i})$ be the copy of
$G_{2}$ corresponding to the $i$-th vertex of $G_{1}$. We define
$V_{i}=\{v_{1}^{i},v_{2}^{i},v_{3}^{i}, \cdots, v_{n}^{i}\}$ where
$v_{j}^{i}$ denotes the $j$-th vertex of the copy $G_{2}$
corresponding to the $i$-th vertex of $G_{1}$. For any integer
$k\geq 2$, the graph $G_1\odot^k G_2$ is recursively defined from
$G_1\odot G_2$ as $G_1\odot^k G_2$ = $(G_1\odot^{k-1} G_2) \odot
G_2$ and hence the order of $G_1\odot^k G_2$ is $m(n+1)^k$. In this
section, we prove some results on the fixing number of $G_1\odot^k
G_2$. For instance, we show that for a connected graph $G_1$ and an
arbitrary graph $G_2$ with $|V(G_1)|= m$ and $|V(G_2)|=n$, $fix
(G_1\odot G_2)$ = $m fix(G_2)$. Further, $fix (G_1\odot^k G_2)$ =
$m(n+1)^{k-1} fix(G_2)$.

%%% ----------------------------------------------------------------------

%\begin{lem}\label{lma1} \cite{B} If there exists an automorphism $\alpha\in
%\Gamma(G_1)$ such that $\alpha(u)=v$, $ u\neq v$ and $d(u,x)=d(v,x)$
%for some $x\in V(G_1)$ , then $x\notin fix(u,v)$.
%\end{lem}
\begin{lem}\label{obs 01} No asymmetric graph $G$ on $n$ vertices contains two or more vertices of degree $n-1$.
\end{lem}
\begin{proof} Suppose contrarily that there exist two distinct vertices $a,b \in
V(G)$ such that $deg(a)=deg(b)=n-1$, then $N(a)=N(b)$ and hence
there exist an automorphism $\alpha$ such that $\alpha(a)=b$ and
$\alpha(c)=c$ for any other vertex $c$ of $G$, a contradiction.
\end{proof}
\begin{lem} If $G$ is an asymmetric graph on $n$ vertices $G+K_{1}$ is also asymmetric.
\end{lem}

\begin{proof}Let $V(G+K_{1})=V(G)\cup \{v\}$. Suppose contrarily that $G+K_{1}$ is not asymmetric then
there exist at least two distinct vertices $a,b\in V(G+K_{1})$ such
that $a\in \mathcal{O}(b)$. Now there are two cases:
\\Case $1$: $a,b \in V(G)$. Then $a\in \mathcal{O}(b)$ in $G$ leads to a
contradiction.
\\Case $2$: $a=v$ and $b\in V(G)$. Then
$deg_{G+K1}(v)=n=deg_{G+K1}(b)$, which is not possible by Lemma
\ref{obs 01}.
\end{proof}

\begin{thm}\label{thm 03}
Let $G$ be a graph. Then there exists a fixing set $\mathcal{F}$ for
$G + K_1$ such that $\mathcal{F}\subseteq V(G)$.
\end{thm}

\begin{proof} Let $V(G + K_1) = V(G)\cup\{v\}$ and $\mathcal{F}$ be a
fixing set for $G + K_1$. If $v\not\in \mathcal{F}$, then nothing to
prove. Let $v\in \mathcal{F}$. Since $\mathcal{F}$ is a fixing set,
there exists $x\in V(G)$ such that $x \sim v$. Then
$(\mathcal{F}\setminus\{v\}) \cup \{x\}\subseteq V(G)$ is a fixing
set for $G + K_1$.
\end{proof}
\begin{lem}\label{lma3}
 Let $G_1$ be a connected graph and $G_2$
be an arbitrary graph. Then for two distinct vertices $u_{1}, u_{2}
\in V(G_1)$, $v_{j}^{2} \notin \mathcal{F}(C_{1})$ for all
$v_{j}^{2}\in V_{2}$.
\end{lem}
\begin{proof}
Let $v_{p}^{1},v_{q}^{1}\in V_{1}$ such that $u_{p}^{1}\in
\mathcal{O}(v_{q}^{1})$. Since for any $v_{j}^{2}\in V_{2}$,
$d(v_{p}^{1},v_{j}^{2})=d(v_{q}^{1},v_{j}^{2})$. Then by Lemma
\ref{lem1}, any vertex of $C_{1}$ is not fixed by $v_{j}^{2}$.
\end{proof}
For $v_{s}^{i},v_{t}^{i}\in V_{i}$ such that $v_{s}^{i}\sim
v_{t}^{i}$, $d(v_{s}^{i},v_{j}^{r})=d(v_{t}^{i},v_{j}^{r})$ for any
$v_{j}^{r}\in V_{j}$ so by Lemma \ref{lem1}, we have the following
proposition:
\begin{prop}\label{prp2}
Let $G_1=(V,E)$ be a connected graph and $G_2$ be an arbitrary
graph. Let $C_i=(V_i,E_i)$ be the subgraph of $G_1\odot G_2$
corresponding to the $i^{th}$-copy of $G_2$. If
$v_{s}^{i},v_{t}^{i}\in V_i$, then $fix(v_{s}^{i},v_{t}^{i})\cap
V_i\neq \emptyset$ and $fix(v_{s}^{i},v_{t}^{i})\cap V_j =
\emptyset$ for $i\neq j$.
\end{prop}
\begin{lem}\label{lma4}
Let $G_1=(V,E)$ be a connected graph and $G_2$ be an arbitrary graph
with $|V(G_1)|=m$ and $|V(G_2)|=n$. Let $C_i=(V_i,E_i)$ be the
subgraph of $G_1\odot G_2$ corresponding to the $i^{th}$-copy of
$G_2$.
\\
\ (1)\ \ \ \ If $\mathcal{F}$ is a minimum fixing set of $G_1\odot
G_2$, then $\mathcal{F}\cap V_i \neq \emptyset$ for every \par \ \ \
\ \ $i \in \{1,2,3,...,m\}$.
\\
(2)\ \ \ \ If $\mathcal{F}$ is a minimum fixing set of $G_1\odot
G_2$, then $\mathcal{F}\cap V = \emptyset$.
\\
(3)\ \ \ \ If $\mathcal{F}$ is a fixing set of $G_1\odot G_2$ then
$\mathcal{F}\cap V_i$ is a fixing set for $C_i$ for every \par \ \ \
\ \ $i\in \{1,2,3,...,m\}$.
\end{lem}
\begin{proof}
(1)Suppose $\mathcal{F}\cap V_i = \emptyset$, then by Lemma
\ref{lem1}, no vertex of $V_i$ will be fixed by any vertex of $V_j$,
where $i\neq j$. Since no $C_i$ is asymmetric so there exists at
least two vertices $u,v \in V_i$ such that $u\in \mathcal{O}(v)$ and
$v \in \mathcal{O}(u)$, which is a contradiction to the fact that
$\mathcal{F}$ is a fixing set. So, $\mathcal{F}\cap V_i\neq
\emptyset$.
\\
(2) We will show that $\mathcal{F}'=\mathcal{F}-V$ is a fixing set
for $G_1\odot G_2$. By using Theorem \ref{thm 03}, for every vertex
$v_i\in V(G_1)$, there exists a fixing set $\mathcal{F}_i\subseteq
V(G_{2})$ for $\{v_{i}\}+G_{2}$ for each $i\in\{1,2,3,...,m\}$.
Also, $\mathcal{F}=\cup \mathcal{F}_i$ is a fixing set for $G_1\odot
G_2$. Hence, the result.
\\
(3) Assume contrary that $\mathcal{F}\cap V_i$ is not a fixing set
for $C_i$ for some $i$, then there exists $u, v\in V_i$ such that
$u\in \mathcal{O}(v)$ and $u, v\notin \mathcal{F}\cap V_i$, then by
using Lemma \ref{lma3}, $\mathcal{F}$ is not the fixing set for
$G_1\odot G_2$, a contradiction.
\end{proof}
Now we give the fixing number of corona product of a connected and
an arbitrary graph that may or may not be connected and no graph is
asymmetric.
% two graphs $G_1$ and $G_2$
\begin{thm}\label{lma5}
\par Let $G_1$ be a connected graph and $G_2$ be an arbitrary graph with $|V(G_1)|=m$ and
$|V(G_2)|=n$. Then $fix(G_1\odot G_2)= m fix(G_2)$.
\end{thm}

\begin{proof}
\par Let $\mathcal{F}$ be a fixing set of $G_1\odot G_2$. From Lemma \ref{lma4}
 $(2)$, we see that $\mathcal{F}\cap V = \emptyset$. Moreover,
from Lemma \ref{lma4} $(1)$, for every $i\in \{1,2,3,...,m\}$, there
exist a non empty set $\mathcal{F}_i  \subset V_i$ such that
$\mathcal{F}=\bigcup \limits_{i=1}^{m} \mathcal{F}_i$, and by Lemma
\ref{lma4} $(3)$, $\mathcal{F}_i$ is a fixing set for $C_{i}$. So,
$fix(G_1\odot G_2)=|\mathcal{F}|= \sum
\limits_{i=1}^{m}|\mathcal{F}_i|= \sum\limits_{i=1}^{m} fix(G_2)=m
fix(G_2)$.
\end{proof}
Since $G_1\odot ^k G_2=(G_1\odot ^{k-1}G_{2})\odot G_{2}$ for any
positive integer $k$ $i.e$ we join a copy of $G_{2}$ at each vertex
in $G_1\odot ^{k-1} G_2$ and $|G_1\odot ^{k-1} G_2|=m(n+1)^{k-1}$.
Therefore we have $m(n+1)^{k-1}$ new copies in $k$-th corona and it
suffices to fix these copies in order to fix the graph $G_1\odot ^k
G_2$. Therefore we have the following corollary:
\begin{cor}\label{thm1}
\par Let $G_1$ be a connected graph and $G_2$ be an arbitrary graph with $|V(G_1)|=m$ and
$|V(G_2)|=n$. Then $fix(G_1\odot ^k G_2)= m(n+1)^{k-1}fix(G_2)$.
\end{cor}
%\begin{rem}\label{rma1}
%The join graph $N_{m}+P_n$ is called the fan graph $F_{m,n}$, where
%$N_m$ is the empty graph of order $m$ and $P_n$ is the path graph of
%order $n$. $F_{1,n}$ corresponds to usual fan graph. Then
%$$fix(F_{1,n})=\left\{\begin{array}{ll}
%       1 & \,\,\,\,\,\,\, \mbox{for}\,\ n=1,\\
 %      2 & \,\,\,\,\,\,\, \mbox{for}\,\ n=2,3,\\
  %     1 & \,\,\,\,\,\,\, \mbox{for}\,\ n\geq 4.
   %         \end{array}
    %         \right.
 %$$
%\end{rem}

\begin{lem}\label{lma6}
If $G_1$ and $G_2$ are asymmetric graphs, then $G_1\odot G_2$ is
also asymmetric.
\end{lem}
\begin{proof} Note that, the graph induced by $V_{i}$ is $C_{i}$ and $C_{i}\cong G_{2}$, hence $C_{i}$ is asymmetric for each $i=1,2,3,\cdots, m$.
Also, note that, $C_{i}\cong C_{j}$ for any $i,j \in V(G_{1})$.
Since $G_{1}$ is asymmetric, so $i\not\in \mathcal{O}(j)$ for any
$i,j \in V(G_{1})$ in $G_1\odot G_2$. Hence, $G_1 \odot G_2$ is
asymmetric.
\end{proof}
Now we give the fixing number of corona product of two graphs $G_1$
and $G_2$ that may or may not be asymmetric where $G_1$ is connected
and $G_2$ may or may not be.
\begin{thm} Let $G_1$ and $G_2$ be two arbitrary graphs with $G_{1}$ is connected and $|V(G_1)|=m$ and
$|V(G_2)|=n$. Then $fix(G_1\odot G_2)= max\{fix(G_1), m fix(G_2)\}$.
\end{thm}
\begin{proof}
\par We have three cases:
\\Case $1$: $G_1$ and $G_2$ are asymmetric.
\\Since $G_1$ and $G_2$ are asymmetric so by Lemma \ref{lma6},
$G_1\odot G_2$ is also asymmetric. Hence, $fix(G_1\odot G_2)=
max\{fix(G_1), m fix(G_2)\}=0$.
 \\Case $2$: Only $G_2$ is asymmetric.
\\Since $G_2$ is asymmetric, so $fix(G_2)=0$. Hence,
$fix(G_1\odot G_2)=
 max\{fix(G_1),$
$m fix(G_2)\}$ = $fix(G_1)$. Now we justify that $fix(G_1\odot G_2)
=fix(G_1)$. Note that the graph induced by $V_{i}$ is $C_{i}$ and
$C_{i}\cong G_{2}$ hence $C_{i}$ is asymmetric for each
$i=1,2,3,\cdots, m$. Since $G_{1}$ is not asymmetric and $C_{i}\cong
C_{j}$ for any $i,j\in V(G_{1})$ so there exists an automorphism
$\alpha \in \Gamma(G_1\odot G_2)$ such that $\alpha(i)=j$ and
$\alpha(V_{i})=V_{j}$. Hence, $i\in \mathcal{O}(j)$ in $G_1\odot
G_2$ if and only if $i\in \mathcal{O}(j)$ in $G_{1}$. Hence,
$fix(G_1 \odot G_2)= fix(G_1)$. Therefore $fix(G_1\odot G_2)$ $=
 max\{fix(G_1),$ $m fix(G_2)\}= fix(G_1)$.
 \\Case $3$: Both $G_1$ and $G_2$ are not asymmetric.
\\Since $fix(G_1)<m$, also $G_2$ is not asymmetric so $fix(G_2)\neq
  0$. Hence $max\{fix(G_1), m fix(G_2)\}$ = $m fix(G_2)$
i.e. we have to show that $fix(G_1\odot G_2)= m fix(G_2)$. By Lemma
\ref{lma5},  $fix(G_1\odot G_2)= m fix(G_2)$. Hence, the result.
\end{proof}

\end{document}